\documentclass[oneside,english,11pt]{amsart}
\usepackage{color}
\usepackage{enumerate}
\usepackage{graphicx}

\newcommand{\co}{\mathbb{C}}

\newcommand{\qe}{\mathbb{Q}}
\newcommand{\ze}{\mathbb{Z}}

\newcommand{\cl}[1]{\mathcal{#1}}

\newcommand{\cpt}[1]{\mathbb{C}P^{2}}

\newcommand{\pcn}[1]{\mathbb{P}^{#1}_{ \mathbb{C}}}

\newcommand{\pe}{\mathbb{P}}
\newcommand{\D}{\mathcal{D}}
\newcommand{\B}{\mathcal{B}}

\newcommand{\Ftilde}{\tilde{\mathcal{F}}}

\usepackage{amsmath}

\newcommand{\val}{\mbox{Val}}

\newcommand{\iso}{\mbox{Iso}}
\newcommand{\dic}{\mbox{Dic}}
\usepackage{amsmath,amssymb,amsfonts,amsthm}
\newtheorem{theorem}{Theorem}[section]

\newcommand{\sing}{{\rm Sing}}

\newtheorem{proposition}[theorem]{Proposition}
\newtheorem{corollary}[theorem]{Corollary}
\newtheorem{lemma}[theorem]{Lemma}
\theoremstyle{definition}
\newtheorem{definition}[theorem]{Definition}
\newtheorem{example}[theorem]{Example}

\newtheorem{maintheoremI}{Theorem}

\newtheorem*{theorem*}{Theorem}

\newcommand{\res}{\text{\rm Res}}

\newcommand{\tr}{\text{tr}}

\newcommand{\ord}{\text{ord}}

\newcommand{\mc}[1]{\mathcal{#1}}

\newcommand{\F}{\mathcal{F}}
\newcommand{\C}{\mathbb{C}^2}

\newcommand{\sep}{\mbox{\rm Sep}}
\newcommand{\divv}{\mbox{\rm Div}}
\newcommand{\var}{\mbox{\rm Var}}
\newcommand{\gsv}{\mbox{\rm GSV}}
\newcommand{\bb}{\mbox{\rm BB}}
\newcommand{\cs}{\mbox{\rm CS}}
\newcommand{\I}{\mbox{\rm I}}

\begin{document}
\title[Residue-type indices and holomorphic foliations]{Residue-type indices and holomorphic foliations}
\author{Arturo Fern\' andez-P\' erez $\&$ Rog\'erio Mol}
\address{Departamento de Matem\'atica - ICEX, Universidade Federal de Minas Gerais, UFMG}
\curraddr{Av. Ant\^onio Carlos 6627, 31270-901, Belo Horizonte-MG, Brasil.}
\email{arturofp@mat.ufmg.br $\&$ rmol@ufmg.br}
\thanks{Work supported by MATH-AmSud Project CNRS/CAPES/Concytec and by Universal/CNPq.}

\begin{abstract}
 We investigate residue-type indices for germs of   holomorphic foliations in the plane and characterize second type foliations --- those not containing tangent   saddle-nodes in the reduction of singularities ---
 by an expression involving the Baum-Bott,  variation and   polar excess indices. These local results are
 applied in the
 study of logarithmic foliations on compact  complex   surfaces.
\end{abstract}
\maketitle
\section{Introduction}
In 1997 M. Brunella \cite{brunella1997II} proved the following result:
\begin{theorem*}\label{brunella_index}
Let $\F$ be a  non-dicritical germ of holomorphic foliation at $(\C,p)$ and
 let   $S$ denote   the union of all its  separatrices. If $\F$ is a generalized curve foliation then
\[
\bb_{p}(\F)= \cs_{p}(\F,S) \ \ \text{and} \ \  \gsv_{p}(\F,S) = 0.\]
\end{theorem*}

The foliation $\F$ is said to be a \textit{generalized curve}   if there are no saddle-nodes in its reduction of singularities. This concept was introduced in \cite{camacho1984} and delimits a family of foliations whose topology is closely related to that of their separatrices ---   local invariant curves --- which in this case are all analytic. In the statement of the theorem,
\textit{non-dicritical} means that the separatrices are finite in number. Further, $\bb$, $\cs$ and
$\gsv$ stand for, respectively, the Baum-Bott, the Camacho-Sad and the Gomez-Mont-Seade-Verjovsky
indices.

Generalized curve foliations are part of the broader  family of \textit{second type foliations}, introduced
by J.-F. Mattei and E. Salem in \cite{mattei2004}. Foliations in this family may admit saddle-nodes in the reduction of singularities, provided they are not \textit{tangent saddle-nodes} (Definition
\ref{def-2ndclass}). A second type foliation  satisfies the  remarkable property of getting reduced
once its   set of separatrices --- including the formal ones --- is desingularized.
 Recently,   second type foliations have been the object of some works. We should mention
 \cite{genzmer2007} --- which deals with the ``realization problem'', that is, the existence of foliations with prescribed
reduction of singularities and projective holonomy representations,   \cite{genzmer2016} --- which studies local polar invariants  and applications to the study of the Poincar\'e problem for foliations --- and
 \cite{mol2016} --- where equisingularitiy properties are considered.
 Our main goal in this article is to  give a characterization of second type foliations by means of  residue-type indices, providing a generalization of  Brunella's result.

Our work is strongly based on the notion of \textit{balanced set} or
\textit{balanced equation} of se\-pa\-ra\-tri\-ces (\cite{genzmer2007} and Definition \ref{def-balanced-set}). This   is a geometric objet formed by a finite set of separatrices with weights --- possibly negative, corresponding to poles --- that, in the non-dicritical case, coincides with the whole separatrix set. A balanced set of
separatrices provides a control of the algebraic multiplicity
of the foliation and, for second type foliations, it actually determines it (Proposition
\ref{prop:Equa-Ba}).
In the text, we will preferably see this object as a divisor of formal curves $\B$ --- a  \textit{balanced divisor}
of separatrices --- having a decomposition $\B = \B_{0} - \B_{\infty}$ as a difference of
effective divisors of zeros and poles.

\par To a germ of foliation $\F$   and a finite set of separatrices $C$ --- which can contain purely formal ones --- we associate a triplet of residue-type indices:  the afore mentioned $\cs$-index and   $\gsv$-index, along with the \textit{variation} index $\var $ --- that turns out to be the sum of
 the two fist indices (definitions  in \cite{camacho1982}, \cite{gomezmont1991} and
 \cite{khanedani1997}; see equation \eqref{index-sum} below).
 We then form a quadruplet of indices by including
the \textit{polar excess} index $\Delta$ (\cite{genzmer2016} and Definition \ref{polar_index}).
This one is  calculated by means of polar invariants and   can be seen as a measure of the existence of saddle-nodes in the reduction process of $\F$ (Theorem \ref{polarexcess-zero} and Propostion \ref{total-delta-zero}).
All these indices are subject of a more detailed  discussion in Section \ref{section-indices}.

 Let $\I_{p}(\F,C)$ denote some   index  in the quadruplet.
In the non-dicritical case, if $C$ is the curve formed by the complete set of separatrices, the index is  said to be \textit{total} and is denoted   as $\I_{p}(\F)$.
We extend the notion of   total
index  to dicritical foliations, employing a balanced divisor of separatrices $\B = \B_{0} - \B_{\infty}$
in  place of the curve $C$ in the following way (Definition \ref{def-total-index}):
 $$\I_{p}(\F) := \I_p(\F,\B_{0})-\I_p(\F,B_{\infty}).$$
This definition is particularly well suited to the $\var$-index  and to the $\Delta$-index, since both of
them are additive in the separatrix set.

The main  result of this article is the following:
\begin{maintheoremI}
\label{main-theorem}
\textit{
 Let $\F$ be a germ of holomorphic foliation at $(\C,p)$. Then
 $\F$ is of second type  if and only if
$$\bb_{p}(\F)=\var_{p}(\F)+\Delta_{p}(\F),$$
where $\bb_{p}(\F)$ is the Baum-Bott index,  $\var_{p}(\F)$ and $\Delta_{p}(\F)$ are the
total variation and polar excess indices.
Moreover, $\F$ is a generalized curve foliation if and only if
 $$\bb_{p}(\F)=\var_p(\F).$$
}\end{maintheoremI}
Indeed, for an arbitrary foliation, we can evaluate the difference of the left and right sides of the expression in  the theorem  as a non-negative integer that assembles the contribution of tangent saddle-nodes along the reduction of singularities.
 This is done in Theorem \ref{theorem-bb-zeta}, from which   Theorem \ref{main-theorem} is a corollary.
In the non-dicritical case, $\Delta_{p}(\F) = \gsv_{p}(\F)$ (Theorem \ref{polarexcess-gsv}) and $\F$ s a generalized curve foliation if and only
if $\gsv_{p}(\F)=0$. Theorem \ref{main-theorem} thus recovers the statement of Brunella's theorem
simultaneously providing its converse: a non-dicritical
$\F$ is a  generalized curve foliation if and only if $\bb_{p}(\F)=\cs_p(\F)$.

The article is structured as follows. In   section \ref{basic-section}
  we present some
basic definitions and properties of local foliations with a  specific view on  second type  foliations.
Section \ref{section-indices} is a brief review   on residue-type indices, where
we explain the case of   formal separatrices  and define the total index.
In section \ref{second-variation-section} we introduce a new invariant, the \textit{second variation} index --- the sum of the variation and  polar excess indices --- and  calculate its change
by blow-up maps. Then,
in section \ref{proof-theoremI-section}, we compare second variation and Baum-Bott indices
(Theorem \ref{theorem-bb-zeta}) and derive the proof of Theorem \ref{main-theorem}. Next,
as an application of Theorem \ref{main-theorem},
 we obtain in section \ref{logarithmic-section}
a characterization of non-dicritical logarithmic foliations 
in terms of second type foliation, both in the complex projective plane (Proposition \ref{prop-logarithmic}) and in the more general setting of projective surfaces with infinite cyclic Picard group (Proposition \ref{prop-logarightmic-picZ}).
We close this article by presenting,   in section \ref{examples-section},   numerical data of a pair of examples.

 \section{Basic definitions and notation}
 \label{basic-section}

\par  In order to fix a terminology and a notation, we recall  some basic concepts of local foliation theory. Let $\F$ be a holomorphic foliation with isolated singularities on a complex surface $X$. Let $p\in X$ be a singular point of $\F$. In local coordinates $(x,y)$ centered at $p$, the foliation is given by an analytic vector field
\begin{equation}
\label{vectorfield}
v=F(x,y)\frac{\partial}{\partial{x}}+G(x,y)\frac{\partial}{\partial{y}}
\end{equation}
or by its dual $1-$form
\begin{equation}
\label{oneform}
\omega = G(x,y) d x - F(x,y) d y,
\end{equation}
where  $F, G   \in {\mathbb C}\{x,y\}$ are relatively prime.

A  \textit{separatrix} for  $\cl{F}$ is an invariant formal irreducible curve, that is,
an object given
 by an irreducible formal series
$f\in {\mathbb C}[[x,y]]$   satisfying
$$
\omega\wedge df=(fh) dx\wedge dy
$$
for some formal series
$h\in {\mathbb C}[[x,y]]$.
The separatrix is said to be {\em analytic} or {\em convergent} if
we can take $f\in {\mathbb C}\{x,y\}$. It is said to be \textit{purely formal} otherwise.
We denote by $\sep_p(\F)$ the set of all separatrices   of $\F$ at $p$.

We say that $p \in \co^{2}$  is a {\em reduced} or {\em simple} singularity for $\cl{F}$  if the
 linear part ${\rm D} v(p)$ of the vector field $v$ in \eqref{vectorfield}  is non-zero and  has eigenvalues $\lambda_1, \lambda_2 \in \co$ fitting in one of the two cases:
\begin{enumerate}[(i)]
\item  $\lambda_1 \lambda_2 \neq 0$ and $\lambda_1 / \lambda_2 \not \in \qe^+$  \
({\em non-degenerate} or {\em complex hyperbolic} singularity). \smallskip
\item $\lambda_1 \neq 0$ and  $\lambda_2= 0$ \ ({\em saddle-node} singularity).
\end{enumerate}

\par In case (i), there are analytic coordinates in $(x,y)$ in which $\F$ is induced by the equation
\begin{equation}
\label{non-degenerate}
\omega=x(\lambda_1+a(x,y))dy-y(\lambda_2+b(x,y))dx,
\end{equation}
where $a,b  \in {\mathbb C}\{x,y\}$ are non-unities, so that  $\sep_p(\F)$ is formed by two
transversal analytic branches given by $\{x=0\}$ and $\{y=0\}$. In case (ii), up to a formal change of coordinates, the  saddle-node singularity is given by a $1-$form of the type
\begin{equation}
\label{saddle-node-formal}
\omega = y(1 + \lambda x^{k})dx + x^{k+1} dy,
\end{equation}
where $\lambda \in \mathbb{C}$ and $k \in \mathbb{Z}_{>0}$ are formal invariants  \cite{martinet1982}.
The curve $\{x=0\}$   is an analytic separatrix, called {\em strong}, whereas $\{y=0\}$  corresponds to a possibly formal separatrix, called {\em weak} or {\em central}.
The integer $k+1 > 1$ is called the {\em tangency index} of $\F$ with respect to the
weak separatrix,   {\em weak index} for short, and will be denoted as
$\text{Ind}_{p}^{w}(\F)$.

 Let  $\pi:(\tilde{X},\D)\to(\C,p)$ be a composition of blow-up maps. The divisor  $\D$
is a finite union of components which are  embedded projective lines, crossing normally at {\em corners}.
If  $\F$ is the foliation defined by the $1-$form $\omega$, we denote by
$\tilde{\F} = \pi^{*} \cl{F}$ the {\em strict transform} of $\cl{F}$, the germ of foliation on $(\tilde{X},\D)$ defined locally by $\pi^{*} \omega$, obtained after cancelling the one-dimensional singular components. For a uniform analysis, we include the possibility of $\pi$ being  the identity map and, abusing notation, we set in this case $\tilde{X} = \co^2$,
$\D = \{p\}$ and $\Ftilde = \F$.

With respect to the the divisor $\D$,
   the foliation $\Ftilde$  at a point $q \in \D$  can be:
\begin{itemize}
\item {\em regular},    if there
are local analytic coordinates $(x,y)$ at $q$ such that $\D\subset \{xy=0\}$
and   $\Ftilde: dx = 0$;
\item {\em singular}, if it is not regular;
\item   {\em reduced} or {\em simple},    if
  $q$ is a reduced singularity for $\Ftilde$
and $\D \subset \sep_{q}(\Ftilde)$.
\end{itemize}
For simplicity, we employ the terminology $\D$-regular, $\D$-singular and $\D$-reduced. When $\D = \{p\}$, these notions coincide with
the ordinary concepts of regular point, singular point and reduced singularity.
We say that  $\pi: (\tilde{X},\D) \to
(\co^2,p)$ is   a  {\em reduction of singularities} or {\em desingularization}  for   $\cl{F}$ if
  all points $q \in \D$ are either  $\D$-regular or $\D$-reduced singularities.
There always exists a reduction of singularities
\cite{seidenbeg1968,camacho1984}.
Besides,   there exists a {\em minimal}  one, in the sense that it factorizes
any other reduction of singularities  by an additional sequence of blow-ups.
All along this text,   reductions of singularities are supposed to be minimal.

Given a germ of foliation $\F$ at $(\C,p)$ we introduce
 the set $\cl{I}_{p}(\F)$ of \textit{infinitely near points} of $\F$ at $p$. This  is
defined   in a  recursive way along the reduction of singularities of $\F$.
We do as follows. Given a sequence of blow-ups $\pi: (\tilde{X},\D) \to
(\co^2,0)$ --- an intermediate step in the reduction process --- and a point $q \in \D$ we set:
\begin{itemize}
\item  if $\tilde{\F}$ is $\D$-reduced at $q$, then $\cl{I}_{q}(\tilde{\F}) = \{q\}$;
\item  if  $\tilde{\F}$ is $\D$-singular but not $\D$-reduced at $q$,
we perform a  blow-up $\sigma:( \hat{X}  ,\hat{\D}) \to (\tilde{X},\D)$ at $q$, where $\hat{\D} = \sigma^{-1}(\D) = (\sigma^{*}\D ) \cup D$ and
  $D = \sigma^{-1}(q)$.
If $q_{1}, \ldots, q_{\ell}$ are all $\hat{\D}$-singular points of $\hat{\F} = \sigma^{*} \tilde{\F}$ on $D$,
then $$\cl{I}_{q}(\tilde{\F}) = \{q\} \cup \cl{I}_{q_{1}}(\hat{\F}) \cup \ldots \cup \cl{I}_{q_{\ell}}(\hat{\F}).$$
\end{itemize}
In order to simplify notation, we settle that  a  numerical invariant for a foliation $\F$  at $q \in \cl{I}_{p}(\F)$    actually means the same invariant computed
  for the transform of $\F$ at $q$. Context will make this clear.

For a fixed a reduction process $\pi:(\tilde{X},\D)\to(\C,p)$ for $\F$, a component  $D \subset \D$ can be:
\begin{itemize}
\item {\em non-dicritical}, if $D$ is $\tilde{\F}$-invariant. In this case, $D$ contains a finite number of simple singularities. Each non-corner singularity carries a separatrix   transversal to $D$, whose projection by $\pi$ is a curve in $\sep_{p}(\cl{F})$.
\item {\em dicritical}, if $D$ is not $\tilde{\F}$-invariant. The definition
of reduction of singularities gives that $D$ may intersect only non-dicritical components and that $\tilde{\F}$ is everywhere transverse do $D$. The $\pi$-image of a local leaf of $\tilde{\F}$ at each non-corner point of $D$ belongs to $\sep_{p}(\cl{F})$.
\end{itemize}
Denote by   $\sep_{p}(D) \subset \sep_{p}(\cl{F})$ the set of separatrices whose transforms
 by $\pi$ intersect the
component $D \subset \D$. If $B \in \sep_{p}(D)$ with $D$ non-dicritical, $B$ is said to be \textit{isolated}. Otherwise, it is said to be a \textit{dicritical separatrix}.
This engenders the   decomposition $\sep_{p}(\cl{F}) = \iso_{p}(\cl{F}) \cup \dic_{p}(\cl{F})$, where notations are self evident.
The set $\iso_{p}(\cl{F})$  is finite and  contains all   purely
formal separatrices. It   subdivides further    in two classes:
 \textit{weak} separatrices --- those arising from the weak separatrices of saddle-nodes --- and \textit{strong} separatrices --- corresponding to strong separatrices
of saddle-nodes and separatrices of non-degenerate singularities. On the other hand, if non-empty, $\dic_{p}(\cl{F})$ is an infinite set of analytic
separatrices. A foliation  $\cl{F}$ is said to be {\em  dicritical}
  when $\sep_{p}(\cl{F})$ is infinite, which is equivalent to saying that $\dic_{p}(\cl{F})$ is non-empty. Otherwise, $\cl{F}$ is called {\em non-dicritical}.

Along the text, we would rather adopt the language of \textit{divisors} of formal curves.
More specifically, a \textit{divisor of separatrices} for a foliation $\F$ at $(\C,p)$ is
a formal sum
\[\B = \sum_{B \in \text{Sep}_{p}(\F)} a_{B} \cdot B \]
where the coefficients $a_{B} \in \ze$ are zero except for finitely many $B \in \sep_{p}(\F)$.
We denote by $\divv_{p}(\F)$ the set of all these divisors, which turns into a group with the canonical additive structure.
We follow  the usual terminology and notation:
\begin{itemize}
\item $\B \geq 0$ denotes an \textit{effective} divisor, one whose  coefficients are all  non-negative;
\item   there is a unique decomposition $\B = \B_{0} - \B_{\infty}$, where $\B_{0}, \B_{\infty} \geq 0$ are respectively the \textit{zero}
and \textit{pole} divisors of $\B$;
\item the \textit{algebraic multiplicity} of   $\B$ is
$\nu_{p}(\B) =  \sum_{B \in \text{Sep}_{p}(\F)} a_{B}.$
\end{itemize}
Given a    formal meromorphic equation $\hat{F}$, whose irreducible components define
 separatrices  $B_i$ with multiplicities $\nu_{i}$, we associate   the divisor
$ (\hat{F}) = \sum_{i} \nu_{i} \cdot B_{i}$.
A curve of separatrices $C$, associated to a reduced equation $\hat{F}$, is identified to
the divisor $C = (\hat{F})$. Such an effective divisor is named \textit{reduced}, that is,
all coefficients are either $0$ or $1$. In general,  $\B \in \divv_{p}(\F)$  is reduced if both
$\B_{0}$ and $\B_{\infty}$ are reduced effective divisors.
A divisor $\B$ is said to be \textit{adapted} to a curve of separatrices $C$ if $\B_0 - C \geq 0$.
Finally, the usual intersection number for formal curves at $(\C,p)$, denoted by $(\,\cdot\, , \,\cdot\, )_{p}$,
is canonically extended in a bilinear way to divisors of curves.

Let $\F$ be a germ of foliation at $(\C,p)$ with reduction process $\pi:(\tilde{X},\D)\to (\C,p)$ and let $\tilde{\F} = \pi^{*} \F$ be the strict transform foliation.
A saddle-node  singularity $q \in \sing(\Ftilde)$ is
is said to be a \textit{tangent saddle-node} if  its   weak separatrix is contained in the exceptional divisor $\D$.
We have the following definition  \cite{mattei2004}:

\begin{definition}
\label{def-2ndclass}
 A foliation   is in the \textit{second class} or is \textit{of second type} if there
are no    tangent saddle-nodes in its reduction process.
\end{definition}

 Given a
a component $D \subset \D$, we denote by  $\rho(D)$ its multiplicity,    which coincides with
the algebraic multiplicity of a curve $\gamma$ at $(\mathbb{C}^{2},p)$ whose transform $\pi^{*} \gamma$ meets $D$ transversally
outside a corner of $\D$. The following invariant is a measure of the existence of tangent saddle-nodes in the reduction of singularities of a foliation:

\begin{definition}{\rm
 The \emph{tangency excess} of $\F$   is the number
\[\tau_{p}(\F)=\sum_{q \in \textsl{SN}(\F) }\rho(D_{q})(\text{Ind}_{q}^{w}(\Ftilde)  -1),\]
where  $\textsl{SN}(\F)$ stands for  the set of tangent saddle-nodes  on $\D$ and, if $q \in \textsl{SN}(\F)$, we denote by  $D_q$  the component of $\D$ containing its weak separatrix and by
 $\text{Ind}_{q}^{w}(\Ftilde) > 1$ its weak index.
}\end{definition}
Off course, $\tau_{p}(\F) \geq 0$ and, by definition, $\tau_{p}(\F) = 0$ if and only if $\textsl{SN}(\F) = \emptyset$, that is, if and only if  $\F$ is of second type.
We introduce the following object  \cite{genzmer2007,genzmer2016}:
\begin{definition}{\rm
\label{def-balanced-set}
  A \emph{balanced divisor of separatrices}  for $\F$  is a divisor of the form
\[ \B \ = \
\sum_{B\in {\rm Iso}_p(\F)} B+ \sum_{B\in {\rm Dic}_{p}(\F)}\ a_{B}  \cdot B,\]
where the coefficients $a_{B} \in \mathbb{Z}$  are  non-zero except for finitely many $B \in
\dic_{p}(\F)$, and, for each  dicritical  component $D \subset \D$,
 the following equality is respected:
\[\sum_{B \in {\text{Sep}_{p}(D)}}a_{B} = 2- \val(D).\]
  The integer $\val(D)$ stands for the {\em valence} of a component $D \subset \D$ in the reduction process, that is,  it is the   number of  components of $\D$ intersecting $D$ other from $D$ itself.
}\end{definition}

A balanced divisor $\B$ is called \textit{primitive} if, for every dicritical component $D \in \D$ and
every $B \in \sep_{p}(D)$, either $-1 \leq a_{B} \leq 0$ or $0 \leq a_{B} \leq 1$.
Recall that a balanced divisor $\B$ is \textit{adapted} to a curve of separatrices $C$ if $\B_{0} - C \geq 0$.
A \textit{balanced equation of separatrices} is a formal meromorphic function $\hat{F}$ whose associated divisor is a balanced divisor. A balanced equation is \textit{reduced}, \textit{primitive} or \textit{adapted} to a curve $C$ if the same is true for the underlying divisor.

The tangency excess   measures the extent that a balanced divisor of separatrices computes the
 algebraic multiplicity, as expressed in the following result \cite{genzmer2007}:

\begin{proposition}
\label{prop:Equa-Ba} Let $\F$ be a germ of singular foliation at $(\C,p)$ with $\B$ as a balanced divisor of separatrices.   Denote by $\nu_{p}(\F)$ and $\nu_{p}(\B)$ their
algebraic multiplicities. Then
\[\nu_{p}(\F)=\nu_{p}(\B)-1+\tau_{p}(\F).\]
Moreover,
\[\nu_{p}(\F)=\nu_{p}(\B)-1\]
if and only if $\F$ is a second type foliation.
\end{proposition}

\section{Indices of foliations}
\label{section-indices}
In this section we briefly recall definitions and main properties of  some indices associated to  singular  plane   foliations, following the presentation in \cite{brunella1997II}. Some of these indices are calculated with respect to  invariant analytic curves  and we   explain   how  to extend  their definitions to   formal invariant curves.
We shall also present the   \textit{polar excess} index, introduced in \cite{genzmer2016}.
In our exposition, invariant curves are identified   with   reduced divisors of separatrices.
Calculations and definitions   apply  to germs of foliations   lying   on a complex surface, but    we can transfer them   to the complex plane by taking  local analytic coordinates.
\subsection{The Baum-Bott index}
Let $\F$ be a germ of foliation   defined either by a holomorphic  vector field $v$ as in \eqref{vectorfield} or by a holomorphic $1-$form $\omega$ as
in  \eqref{oneform}.
If $J(x,y)$ denotes the Jacobian matrix  of   $(F,G)$ in the variables $(x,y)$, then the following residue  defines the \textit{Baum-Bott index} at $p\in\sing(\F)$ \cite{baum1972}:
$$\bb_{p}(\F)=\res_{p}\left\{\frac{(\tr J)^2}{F\cdot G}dx\wedge dy\right\}.$$
For a reduced   singularity   with local models \eqref{non-degenerate} and \eqref{saddle-node-formal}, this becomes:
\begin{equation}
\bb_{p}(\F) =
\begin{cases} \displaystyle \frac{(\tr J(p))^2}{\det J(p)}=\frac{\lambda_1}{\lambda_2}+\frac{\lambda_2}{\lambda_1}+2 & \text{if $p$ is non-degenerate};
\medskip \\
2k+2+\lambda & \text{if $p$ is a saddle-node}.
\end{cases}
\end{equation}

On a compact surface $M$, the sum of Baum-Bott indices of a foliation $\F$ is expressed in terms of
the first Chern class of the normal bundle $N_{\F}$ of the foliation \cite{baum1972,brunella1997I}:
\begin{equation}
\label{baumbott}
 \sum_{p\in\sing(\F)}\bb_{p}(\F)=c^{2}_1(N_{\F}).
\end{equation}

\subsection{The Camacho-Sad index}

\par
Let $C$ be an invariant analytic curve for $\F$   defined by a reduced function $f \in {\mathbb C}\{x,y\}$.
Then there are germs  $g, k   \in {\mathbb C}\{x,y\}$, with $k$ and $f$   relatively prime, and a germ of
analytic   $1-$form $\eta$ such that
\begin{equation}
\label{decomposition}
g\omega=kdf+f\eta
\end{equation}
(see, for instance, \cite{linsneto1988,suwa1998}).
The Camacho-Sad index \cite{camacho1982} is the residue
\begin{equation}
\label{camachosaddefinition}
\cs_{p}(\F,C) = -\frac{1}{2\pi i}\int_{\partial{C}}\frac{1}{k}\eta.
\end{equation}
The integral is over $\partial{C}=C\cap S^3$,  the link of $C$ oriented as the boundary of $C\cap B^4$, where $B^{4}$ is a small ball centered at $0 \in \co^{2}$
and $S^{3} = \partial B^{4}$.
If $C_{1}$ and $C_{2}$ are   $\F$-invariant curves without common components, then the following adjunction formula holds:
\begin{equation}
\label{cs-adjunction}
\cs_{p}(\F,C_{1}+ C_{2}) = \cs_{p}(\F,C_{1} ,p) + \cs_{p}(\F,C_{2},p) + 2(C_{1} \cdot C_{2})_{p}  . \end{equation}

A  decomposition   \eqref{decomposition}
also exists for   a branch of formal separatrix $B$
with formal equation $f  \in {\mathbb C}[[x,y]]$,  yielding $g$, $k$ and $\eta$ as formal objects. In this context, we can extend the definition of the Camacho-Sad index to $B$ by taking $\gamma(T)$, a
Puiseux parametrization for $B$ such that $\gamma(0) = p$, and setting
 \[ \cs_{p}(\F,B) = \res_{t=0} \gamma^{*}  \left( \frac{1}{k}\eta \right). \]
Clearly, when $B$ is convergent, this coincides with
\eqref{camachosaddefinition}.
Finally, the $\cs$-index
   may be defined for a  reducible curve of separatrices   containing some purely formal branches by applying the adjunction formula \eqref{cs-adjunction}.

The following result is known as the Camacho-Sad index Theorem \cite{camacho1982}:
if $C \subset M$ is a compact curve invariant by a foliation $\F$   on a complex surface $M$, then
\begin{equation}
\label{cs-theorem}
\sum_{p\in\sing(\mc{F})\cap C}\cs_{p}(\mc{F},C) = C\cdot C.
\end{equation}

\subsection{The Gomez-Mont-Seade-Verjovsky index}
The decomposition
 \eqref{decomposition} is also used to calculate
the $\gsv$-index (due to Gomez-Mont, Seade
and Verjovsky, \cite{gomezmont1991}) with respect to an $\F$-invariant curve $C$:
$$\gsv_{p}(\F,C)= \frac{1}{2\pi i}\int_{\partial{C}}\frac{g}{k} d \left(\frac{k}{g}\right).$$
The  adjunction formula now reads:
\begin{equation}
\label{gsv-adjunction}
\gsv_{p}(\F,C_{1} + C_{2}) = \gsv_{p}(\F,C_{1}) + \gsv_{p}(\F,C_{2}) - 2(C_{1} \cdot C_{2})_{p}, \end{equation}
where $C_{1}$ and $C_{2}$ are  $\F$-invariant curves without common components.

The extension of this definition to a purely formal branch of separatrix $B$ is done as   previously:
 take $\gamma(T)$ a Puiseux parametrization for $B$ such that $\gamma(0) = p$ and set
 \[ \gsv_{p}(\F,B) = \res_{t=0} \gamma^{*}  \left( \frac{g}{k} d \left(\frac{k}{g}\right) \right) =
 \ord_{t=0} \left( \frac{k}{g} \circ \gamma(t)\right). \]
Then,  use the adjunction formula \eqref{gsv-adjunction} in order to define
the $\gsv$-index for an invariant curve  $C$ containing some purely formal branches.

For the $\gsv$-index, we can also state a result of global nature \cite{brunella1997I}:
if the compact curve $C \subset M$ is   invariant by a foliation $\F$   on a complex surface $M$, then
\begin{equation}
\label{gsv-global-sum}
\sum_{p\in\sing(\mc{F})\cap C}\gsv_{p}(\mc{F},C)= c_1(N_{\mc{F}})\cdot C - C \cdot C.
\end{equation}

\subsection{The variation  index}
Each point $q$ in a small punctured neighborhood of $ p \in \mathbb{C}^2$ is regular for $\F$. Then
 there exists a germ of holomorphic $1-$form $\zeta$ at $q$ such that $d\omega=\zeta\wedge \omega$. If $\zeta'$ is another such $1-$form, we have that $\zeta$ and $\zeta'$ coincide  over every leaf of $\F$. Therefore, in this punctured neighborhood, we can define a multi-valued $1-$form, still denoted by $\zeta$, with single-valued   restriction to each leaf of $\F$,  satisfying the equation
  $$d\omega=\zeta\wedge\omega.$$
   The \textit{variation index} \cite{khanedani1997} for an $\F$-invariant analytic curve $C$ is defined as
$$\var_{p}(\mc{F},C)=\frac{1}{2\pi i}\int_{\partial{C}}\zeta.$$
This index is additive in the separatrices of $\F$:
\begin{equation}
\label{var-adjunction}
\var_{p}(\F,C_{1} + C_{2}) = \var_{p}(\F,C_{1}) + \var_{p}(\F,C_{2})
\end{equation}
whenever $C_{1}$ and $C_{2}$ are    $\F$-invariant curves without common components.
Thus, for a divisor of separatrices $\B = \sum_{B} a_{B} \cdot B$   we can define
\[
 \var_{p}(\F,\B) = \sum_{B} a_{B}\var_{p}(\F,B).\]

For an  analytic invariant curve $C$, we have the relation
\begin{equation}
\label{index-sum}
\var_{p}(\F,C) = \cs_{p}(\F,C)  + \gsv_{p}(\F,C) .
\end{equation}
Now, when it comes to   defining $\var_{p}(\mc{F},B)$ for    a formal branch of separatrix $B$,  the strategy followed for the $\cs$ and the $\gsv$ indices
  is unsuitable, since the $1-$form $\zeta$ does not define a formal object
at $p \in \co^{2}$. However, knowing   $\cs_{p}(\F,B)$ and $\gsv_{p}(\F,B)$ for a formal separatrix $B$, we can adopt
formula \eqref{index-sum} as a definition for $\var_{p}(\mc{F},B)$ and use
\eqref{var-adjunction} in order to compute
 $\var_{p}(\F,C)$ for a multi-branched invariant curve  $C$.

The variation index satisfies a property of global nature expressed in the following terms:
if $\F$ is a foliation on a complex surface $M$ and $C \subset M$ is a compact invariant curve, then
\begin{equation}
\label{suwa}
\sum_{p\in\sing{\mc{F}}\cap C} \var_{p}(\mc{F},C)=c_1(N_{\mc{F}})\cdot C.
\end{equation}

\subsection{The polar excess index}
Let $\eta$ be a formal meromorphic $1-$form with trivial divisor of zeros, written  in coordinates $(x,y)$ as
$$\eta=Pdx+Qdy,$$
where $P,Q$ are formal meromorphic functions.  For $(a:b)\in\mathbb{P}^1$, the polar curve of $\eta$ with respect to $(a:b)$ is the formal curve $\mathcal{P}^{\eta}_{(a,b)}$ associated to the equation $aP+bQ=0$. Let $B$ be an irreducible curve,   not contained in the pole divisor
$(\eta)_{\infty}$, having $\gamma(t)$ as a Puiseux parametrization. We say that $B$   is \textit{invariant} by $\eta$ if $\gamma^{*}\eta \equiv 0$. In this case, we define
the \textit{polar intersection number} of $\eta$ and $B$ at $p$  (see \cite{cano2015,genzmer2016}) as the generic value of
\[ (\mathcal{P}^{\eta},B)_p=(\mathcal{P}^{\eta}_{(a,b)},B)_p = \ord_{t=0}((aP+bQ)\circ\gamma)\]
  for   $(a:b)\in\mathbb{P}^1$.  This is an ingredient for the following definition:
\begin{definition}\label{polar_index}
Let $\F$ be a germ of singular foliation at   $(\C,p)$. Let $B$ be a branch of separatrix and $\hat{F}$ be   a reduced balanced equation of separatrices adapted to $B$.  The \textit{polar excess index} \cite{cano2015,genzmer2016} of $\F$ with respect to $B$ is the integer
\[ \Delta_p(\F,B)= (\mathcal{P}^{\F},B)_p-(\mathcal{P}^{d\hat{F}},B)_p. \]
 For a curve of separatrices $C$, with irreducible factors as $B_{1}, \ldots, B_{r}$, we define the polar excess index in an additive way:
\[ \Delta_p(\F,C)=\sum^{r}_{i=1}\Delta_p(\F,B_i)=\sum^{r}_{i=1}
\left((\mathcal{P}^{\F},B_i)_p-(\mathcal{P}^{d\hat{F}},B_i)_p\right). \]
\end{definition}
This definition is independent of the balanced equation, so, in order to compute the polar excess for a multi-branched
curve,    a balanced equation simultaneously adapted to all its branches can be employed.
The additive  character of the $\Delta$-index
  enables us to extend its definition to an arbitrary divisor $\B = \sum_{B} a_{B} \cdot B$ in $\divv_{p}(\F)$:
\[
 \Delta_{p}(\F,\B) = \sum_{B} a_{B}\Delta_{p}(\F,B).\]

We can also formulate the $\Delta$-index   as the residue of the logarithmic derivative
of the ratio of equations of polar curves for $\F$ and for $d \hat{F}$, where
$\hat{F}$ is an irreducible balanced equation of separatrices adapted to the invariant curve.
More precisely,
if
$\omega = P dx + Q dy$ induces $\F$,
we define,    for $(a:b) \in \pe^{1}$, the formal meromorphic $1-$form
\[ \eta_{(a:b)} = \frac{a \hat{F}_{x} + b \hat{F}_{y}}{a P + b Q}
d \left(\frac{a P + b Q}{a \hat{F}_{x} + b \hat{F}_{y}}\right) = \frac{d(a P + b Q)}{a P + b Q}
-  \frac{d(a \hat{F}_{x} + b \hat{F}_{y})}{a \hat{F}_{x} + b \hat{F}_{y}}.
\]
Then, for generic $(a:b)$,
\[  \Delta_p(\F,B) = \res_{t=0} \gamma^{*} \eta_{(a:b)}. \]
Moreover, if $C$ is an $\F$-invariant analytic curve, then, still for generic $(a:b)$,
$$\Delta_p(\F,C)=\frac{1}{2\pi i}\int_{\partial{C}}\eta_{(a:b)}.$$

The following simple calculations    are done in   \cite{genzmer2016}
for an $\F$-invariant branch $B$:
\begin{itemize}
\item If $\F$ is a non-singular    then $\Delta_p(\F,B)=0.$
\item If $\F$ has a  non-degenerated reduced singularity, then
$$\Delta_p(\F,B)=0.$$
\item If $\F$ has a saddle-node singularity   with weak index $k+1$ we have two possi\-bi\-li\-ties:  either  $\Delta_p(\F,B)=0$, when $B$ is the strong separatrix, or $\Delta_p(\F,B)=k>0$, when $B$ is the weak separatrix.
\end{itemize}
In general, taking into account the behavior of the  $\Delta$-index under blow-ups (equation \eqref{beta_10} below), we have
\begin{equation}
\label{delta-formula}
\Delta_p(\F,B) =
  \begin{cases} \displaystyle \sum_{q \in \mathcal{I}_{p}(\F)} \nu_{q}(B) \tau_{q}(\F) &\text{if $B$ is a strong or dicritical separatrix} \medskip\\
  \displaystyle  k + \sum_{q \in \mathcal{I}_{p}(\F)} \nu_{q}(B) \tau_{q}(\F) &\text{if $B$ is a weak separatrix,}
\end{cases}
\end{equation}
where $k+1$ is the weak index associated to $B$. Thus,  $\Delta_p(\F,B) \geq 0$ and $\Delta_p(\F,B) = 0$ if and only if $\F$ is a second class foliation and $B$ is a strong or dicritical separatrix.
The polar excess index is a  measure of the existence of saddle-nodes singularities in the desingularization of $\F$. This interpretation derives from the following  result of \cite{genzmer2016}, which
is a consequence of formula \eqref{delta-formula}:

\begin{theorem}
\label{polarexcess-zero}
If $\F$ is a germ of singular foliation at $(\C,p)$ and $C$ is a curve of separatrices, then
$\Delta_p(\F,C) \geq 0$. Moreover, if
  $\B$ is a balanced divisor of  separatrices of $\F$, then $\F$ is generalized curve foliation if and only if $$\Delta_p(\F,\B_0)=0.$$
\end{theorem}

The polar excess and the
 $GSV$-index are interrelated by the following result \cite{genzmer2016}:
\begin{theorem}
\label{polarexcess-gsv}
Let $\F$ be a germ of singular foliation at $(\C,p)$. Let $C$ be a curve of separatrices and $\B$ be a balanced divisor adapted to $C$. Then
$$\gsv_p(\F,C)=\Delta_p(\F,C)+(C,\B_0 - C)_p-(C,\B_{\infty})_p.$$
In particular, when
$\F$ is non-dicritical   and $C$ is the complete set of separatrices, then
$$\gsv_p(\F,C)=\Delta_p(\F,C).$$
\end{theorem}

\subsection{The total index}

Let $\I_{p}(\F,C)$ denote one of the four residue-type indices relative to
a curve of separatrices $C$ defined so far --- $\cs$, $\gsv$, $\var$ or $\Delta$.
When $\F$ is non-dicritical and $C$ is the complete set of separatrices, it is usual to say that $\I_{p}(\F,C)$ is \textit{total}. When it comes to dicritical singularities, an attempt to establish
a definition of total index involves the choice
of  a finite subset of $\sep_{p}(\F)$ as a reference. We propose  to use  balanced
divisors of separatrices for this goal:
\begin{definition}
\label{def-total-index}
Let $\F$ be a foliation at $(\C,p)$ and $\B$ be a primitive balanced divisor
of separatrices.  The \textit{total} index of $\F$ at $p$ is
defined as
$$\I_{p}(\F,\B_{0}) - \I_{p}(\F,\B_{\infty})$$
and denoted by
$$\I_{p}(\F) = \cs_{p}(\F),\ \gsv_{p}(\F),\ \var_{p}(\F)\ \text{ or }\ \Delta_{p}(\F).$$
\end{definition}
Observe that $\I_{p}(\F,B)$ is the same for all branches $B \in \sep_{p}(D)$
associated to the same dicritical component $D \subset \D$. This results from
formula \eqref{delta-formula} for the $\Delta$-index and, for the three other indices,
 from similar formulas
 based on their behavior under blow-ups (see \cite{brunella1997II}).
 As a consequence, $\I_{p}(\F)$ does not depend on the choice of the primitive balanced divisor.
We inherit a connecting relation similar to \eqref{index-sum}:
\[ \var_{p}(\F) = \cs_{p}(\F)+ \gsv_{p}(\F) .\]

The total $\var$ and the total  $\Delta$ indices  may be calculated using any balanced divisor
of separatrices $\B$, not necessarily a reduced one:
\[ \var_{p}(\F) = \var_{p}(\F,\B) \quad \text{and} \quad \Delta_{p}(\F) = \Delta_{p}(\F,\B).\]

Next we  state   a slightly modified version of Theorem \ref{polarexcess-zero} involving the total $\Delta$.
We remark that, when the desingularization divisor  $\D$ of $\F$ is devoid of dicritical components of valence two or higher, there are no
 poles in a primitive
balance divisor of separatrices and
the  statement below is precisely that of Theorem \ref{polarexcess-zero}.

\begin{proposition}
\label{total-delta-zero}
$\F$ is a generalized curve foliation at $(\C,p)$ if and only if $\Delta_{p}(\F) = 0$.
\end{proposition}
\begin{proof}
A generalized curve foliation is in particular second class and all its separatrices  are either strong or dicritical. Thus,  formula \eqref{delta-formula} gives
$\Delta_{p}(\F,B) = 0$ for every $B \in \sep_{p}(\F)$,  which on its turn implies that $\Delta_{p}(\F,\B)=0$ for any divisor of separatrices $\B$ and, in particular, for a balanced divisor.

The converse proof is based on
  the following fact: if $\D$ is the desingularization divisor of
$\F$, then there is at least one isolated separatrix crossing each   component of the
$\tilde{\F}$-invariant part of $\D$   \cite[Prop. 4]{mol2002}.
Thus, the number of isolated separatrices
is at least $1  + \sum_{D}(\val(D) - 1)$, where the sum is over all dicritical components $D \subset \D$.
Let $\B$ be a primitive balanced divisor and
$D \subset \D$ be a dicritical component of  $\val(D) > 2$ .
The pole divisor $\B_{\infty}$ contains $\val(D) - 2$ separatrices of $\sep_{p}(D)$.
Note that $D$ appears in the desingularization process as a component of valence 0, 1, or 2 (when it results, respectively, from
the blow-up at $p$ itself, at a non corner singularity or at a corner singularity). So  at least
$\val(D) - 2$ points of $D$ will be blown-up in the subsequent steps of the reduction process and to each one
of them we can associate an isolated separatrix.
Therefore, to each dicritical separatrix $B$ appearing in $\B_{\infty}$, we can associate in an injective way
 one such isolated separatrix $\tilde{B}$. It follows from \eqref{delta-formula}
that
\begin{equation}
\label{delta-inequality}
   \Delta_{p}(\F,B) \leq \Delta_{p}(\F,\tilde{B}).
\end{equation}
Denote by $\tilde{\B}_0$ the divisor obtained
by summing up these $\tilde{B}$. We have $\Delta_{p}(\F, \tilde{\B}_0 - \B_{\infty}) \geq 0$
by \eqref{delta-inequality}.
Now we decompose $\B_0 =  \hat{\B}_0 + \tilde{\B}_0$ as a sum of effective divisor, where $\hat{\B}_0$ is non-trivial.
Then
\[0 = \Delta_{p}(\F,\B) = \Delta_{p}(\F,\B_{0})- \Delta_{p}(\F,\B_{\infty})
= \Delta_{p}(\F,\hat{\B}_0) + \Delta_{p}(\F, \tilde{\B}_0 - \B_{\infty}).\]
The terms at the right are non negative and thus both   are zero. This implies, in particular,
that  $\Delta_{p}(\F,B)= 0$ for every separatrix $B$ in $\hat{\B}_0$.
Formula   \eqref{delta-formula} then gives at once 
that $\F$ is a second type foliation and that every  isolated
$B$ in $\hat{\B}_0$ is a weak separatrix.
For the  separatrices in $\tilde{\B}_{0}$, remark that each inequality \eqref{delta-inequality} is actually an
equality,
and this is possible only if $\tilde{B}$ is a strong separatrix.
Summarizing, $\F$ is a second class foliation having only strong isolated separatrices.
It is therefore a generalized curve foliation.
\end{proof}

\section{Second variation index}
\label{second-variation-section}
In order to condense notation and terminology, we assemble the variation  and the polar excess indices in  a new invariant:
\begin{definition}
Let $\F$ be a germ of singular foliation at $(\C,p)$ and $C$ be a curve of separatrices.  The \textit{second variation index} of $\F$ along $C$ is
defined as
\begin{eqnarray*}
\zeta_p(\F,C)= \var_p(\F,C)+ \Delta_p(\F,C).
\end{eqnarray*}
\end{definition}

The variation and the polar excess are additive in the separatrices and this property is inherited by the $\zeta$-index. We can therefore define it for a divisor of separatrices and   have    a \textit{total second variation index}   by means of a balanced  divisor $\B$:
\[
\zeta_{p}(\F)= \zeta_{p}(\F,\B) =   \var_{p}(\F)+\Delta_{p}(\F).\]

\par Next, we describe the behavior of the second variation under a blow-up $\sigma:(\tilde{\mathbb{C}}^2,D)\to (\C,p)$. As usual,  we denote
respectively by $\sigma^{*}\F = \tilde{\F}$ and $\sigma^{*}B = \tilde{B}$ the transforms of the foliation $\F$ and of a   branch of separatrix $B \in \sep_{p}(\F)$. A divisor of separatrices $\B = \sum_{B} a_{B}\cdot B$ is said to be of \textit{order} $q \in D$ if $\tilde{B}  \cap D = q$ whenever $a_{B} \neq 0$. If this is so, the transform of $\B$ is defined as     $\tilde{\B}= \sum_{B}a_{B} \cdot \tilde{B}$, which
is a divisor of separatrices for $\tilde{\F}$ at $q \in D$.

\begin{lemma}\label{variation_beta}
With the notation above, if $q = \tilde{B} \cap D$, then
$$\zeta_q(\tilde{\F},\tilde{B})=\zeta_p(\F,B)-(m_p(\F)+\tau_p(\F))\nu_p(B),$$
  where  $\tau_p(\F)$ is the   tangency excess of $\F$  at $p$ and
$$m_p(\F)=\begin{cases}
\nu_p(\F),  &\text{if $\sigma$ is non-dicritical}\\
\nu_p(\F)+1, &\text{if $\sigma$ is dicritical}.
\end{cases}$$
Moreover,
if  $\B$  is a divisor of separatrices of order $q \in D$, then
$$\zeta_q(\tilde{\F},\tilde{\B})=\zeta_p(\F,\B)-(m_p(\F)+\tau_p(\F))\nu_p(\B).$$
 \end{lemma}
\begin{proof}
The  formula for a branch $B \subset \sep_{p}(\F)$ is a consequence of known formulas for
the behavior under blow-ups for the variation \cite{brunella1997II} and the polar excess \cite{genzmer2016}  indices:
\begin{equation}
\label{beta_10}
\begin{array}{l}
 \var_q(\tilde{\F},\tilde{B}) = \var_p(\F,B)-m_p(\F)\nu_p(B), \medskip \smallskip \\
 \Delta_q(\tilde{\F},\tilde{B}) = \Delta_p(\F,B)-\tau_p(\F)\nu_p(B).
 \end{array}
 \end{equation}
The expression for a divisor is then a consequence of the additiveness of the second variation index.
\end{proof}

Now we examine   the total second variation.
We have that  $\zeta_{p}(\F) = \zeta_{p}(\F, \B)$, where $\B$ is a balanced divisor of separatrices.
Suppose that the $D$-singular point of $\tilde{\F}$ are $q_1,\ldots,q_{\ell}$. In order to calculate the total $\zeta$ at these points we need to relate the  transform of  $\B$ with balanced divisors
at the points $q_{j}$. Denote by $S(q_{j})  \subset \text{Sep}_{p}(\F)$
 the subset of all separatrices of order $q_{j} \in D$ and decompose
 $$\B = \sum_{B \in \text{Sep}_{p}(\F)} a_{B} \cdot B = \sum_{j=1}^{\ell} \sum_{B \in S(q_{j})} a_{B} \cdot B =  \sum_{j=1}^{\ell} \B_{j},$$
where $\B_{j} = \sum_{B \in S(q_{j})} a_{B} \cdot B$. As before, denote by $\tilde{\B}_{j}$    the transform of  $\B_{j}$.
 There are two situations \cite{genzmer2016}:
\begin{itemize}
\item  $\sigma$ is a non-dicritical blow-up, meaning that the exceptional divisor   is $\tilde{\F}$-invariant. Then $\tilde{\B}_{j} + D$ is  a balanced divisor for $\tilde{\F}$  at $q_{j}$, where
   we keep denoting by $D$ the germ of the exceptional divisor at $q_{j}$.
\item   $\sigma$ is a dicritical blow-up, one such that the exceptional divisor   is not $\tilde{\F}$-invariant. Then $\tilde{\B}_{j}$  is a balanced divisor  for $\tilde{\F}$  at $q_{j}$.
\end{itemize}

We can state the following result:
\begin{proposition}
\label{zeta-blowup}
Let  $\sigma:(\tilde{\mathbb{C}}^2,D)\to (\C,p)$ be a   blow-up at $p\in \C$. Suppose that $q_1,\ldots,q_{\ell}$ are the $D$-singular points of $\tilde{\F}$. Then
\[
\zeta_{p}(\F) =
\begin{cases}
\displaystyle  \sum^{\ell}_{j=1}\zeta_{q_{j}}(\tilde{\F})+\nu^{2}_p(\F)-\tau^{2}_p(\F)
 &\text{if $\sigma$ is non-dicritical} \\
\displaystyle \sum^{\ell}_{j=1}\zeta_{q_{j}}(\tilde{\F})+(\nu_p(\F)+1)^2-\tau^{2}_p(\F)
 &\text{if $\sigma$ is dicritical.}
 \end{cases}
\]
\end{proposition}
\begin{proof}
\par We split the proof in two parts.
\par \noindent \underline{Part 1}: \textit{The non-dicritical case}.
The total $\zeta$ at each $q_{j}$ is
\begin{equation}\label{beta_total}
\zeta_{q_{j}}(\tilde{\F})=\zeta_{q_{j}}(\tilde{\F},\tilde{\B}_j+D)=\zeta_{q_{j}}(\tilde{\F},\tilde{\B}_j)+
\zeta_{q_{j}}(\tilde{\F},D).
\end{equation}
We first calculate
\begin{equation}\label{beta_16}
\sum^{\ell}_{j=1}\zeta_{q_{j}}(\tilde{\F},D)=\sum^{\ell}_{j=1}\var_{q_{j}}(\tilde{\F},D)+\sum^{\ell}_{j=1}\Delta_{q_{j}}(\tilde{\F},D).
\end{equation}
The sum of $\var$-indices along $D$ is given by equation \eqref{suwa}:
\begin{equation}
\label{beta_17}
\sum^{\ell}_{j=1}\var_{q_{j}}(\tilde{\F},D)=\left(c_1(N_{\tilde{\F}})\right)\cdot D=\left(-\nu_p(\F)\,D \right)\cdot D=\nu_p(\F).
\end{equation}
On the other hand, $\tilde{\B}_j + D$ is  a balanced divisor of separatrices at  $q_{j} \in D$. Thus,
 we get from   Theorem \ref{polarexcess-gsv} that
\begin{equation}\label{beta_18}
\sum^{\ell}_{j=1}\Delta_{q_{j}}(\tilde{\F},D)=\sum^{\ell}_{j=1}\gsv_{q_{j}}(\tilde{\F},D)
-\sum^{\ell}_{j=1}(D,\tilde{\B}_j)_{q_{j}}.
\end{equation}
Now, we use \eqref{gsv-global-sum} to compute the sum of $\gsv$-indices along $D$:
\begin{eqnarray*}
\sum^{\ell}_{j=1}\gsv_{q_{j}}(\tilde{\F},D)&=&c_1(N_{\tilde{\F}})\cdot D-D\cdot D\\
&=&(-\nu_p(\F)D)\cdot D-D\cdot D\\
&=&  \nu_p(\F)+1.
\end{eqnarray*}
Since
\begin{equation*}
\sum^{\ell}_{j=1}(D,\tilde{\B}_j)_{q_{j}}=\sum^{\ell}_{j=1}\nu_p(\B_j)=\nu_p(\B),
\end{equation*}
using  Proposition \ref{prop:Equa-Ba},   equation \eqref{beta_18} turns into
\begin{equation}\label{beta_19}
\sum^{\ell}_{j=1}\Delta_{q_{j}}(\tilde{\F},D)=\nu_p(\F)+1-\nu_p(\B)=\tau_p(\F).
\end{equation}
It follows from  \eqref{beta_16}, \eqref{beta_17} and \eqref{beta_19} that
\begin{equation}\label{beta_20}
\sum^{\ell}_{j=1}\zeta_{q_{j}}(\tilde{\F},D)=\nu_p(\F)+\tau_p(\F).
\end{equation}
Combining \eqref{beta_total} and \eqref{beta_20}, we find
\begin{equation}\label{beta_21}
\sum^{\ell}_{j=1}\zeta_{q_{j}}(\tilde{\F})=\sum^{\ell}_{j=1}\zeta_{q_{j}}(\tilde{\F},\tilde{\B}_j)+\nu_p(\F)+\tau_p(\F).
\end{equation}
Now Lemma \ref{variation_beta} implies that
\begin{equation}
\label{beta_22}
\begin{array}{rcl}
\displaystyle \sum^{\ell}_{j=1}\zeta_{q_{j}}(\tilde{\F},\tilde{\B}_j)&=&\sum^{\ell}_{j=1}\zeta_p(\F,\B_j)-
(\nu_p(\F)+\tau_p(\F))\sum^{\ell}_{j=1}\nu_p(\B_j) \medskip\\ \displaystyle
&=&\zeta_{p}(\F)-(\nu_p(\F)+\tau_p(\F))\nu_p(\B).
\end{array}
\end{equation}
From \eqref{beta_21} and \eqref{beta_22},
\begin{eqnarray}
\zeta_{p}(\F)&=&\sum^{\ell}_{j=1}\zeta_{q_{j}}(\tilde{\F})+(\nu_p(\F)+\tau_p(\F))(\nu_p(\B)-1)\nonumber\\
&=&\sum^{\ell}_{j=1}\zeta_{q_{j}}(\tilde{\F})+(\nu_p(\F)+\tau_p(\F))(\nu_p(\F)-\tau_p(\F)) \nonumber \\
&=&\sum^{\ell}_{j=1}\zeta_{q_{j}}(\tilde{\F})+\nu^{2}_p(\F)-\tau^{2}_p(\F) \nonumber
\end{eqnarray}
and we are done.

\par \noindent \underline{Part 2}: \textit{The  dicritical case}.
Now $\tilde{\B}_j$ is  a balanced divisor of separatrices for $\tilde{\F}$ at $q_j$. Then, it  follows from Lemma \ref{variation_beta}  and Proposition \ref{prop:Equa-Ba} that
\[
\begin{array}{rcl}
\zeta_{p}(\F)&=&\displaystyle \sum_{j=1}^{\ell}\zeta_{q_{j}}(\F,\B_{j})  \\
&=& \displaystyle \sum^{\ell}_{j=1}\zeta_{q_{j}}(\tilde{\F},\tilde{\B}_j)+(\nu_p(\F)+1+\tau_p(\F))
\sum^{\ell}_{j=1}\nu_p(\B_j)  \\
&=&\displaystyle \sum^{\ell}_{j=1}\zeta_{q_{j}}(\tilde{\F})+(\nu_p(\F)+1+\tau_p(\F))\nu_p(\B) \\
&=&\displaystyle\sum^{\ell}_{j=1}\zeta_{q_{j}}(\tilde{\F})+(\nu_p(\F)+1+\tau_p(\F))(\nu_p(\F)+1-\tau_p(\F))  \\
&=& \displaystyle \sum^{\ell}_{j=1}\zeta_{q_{j}}(\tilde{\F})+ (\nu_p(\F)+1)^2-\tau^{2}_p(\F).
\end{array} \]
This completes the proof of the proposition.
\end{proof}

\section{Proof of Theorem  \ref{main-theorem}}
\label{proof-theoremI-section}

In this  section we compare     second variation   and   Baum-Bott indices and achieve a proof for
Theorem \ref{main-theorem}. We start with a look at
 reduced singularities:
\begin{lemma}
\label{lemma_reduced}
Let $\F$ be a reduced germ of foliation at $(\C,p)$. Then
$$\zeta_{p}(\F)=\bb_{p}(\F).$$
\end{lemma}
\begin{proof} A reduced foliation is non-dicritical and so  $\Delta_{p}(\F) = \gsv_{p}(\F)$, which implies $\zeta_{p}(\F) = \var_{p}(\F) + \gsv_{p}(\F)$.
We only need to assemble information from \cite{brunella1997II} (see the two examples on p. 538).

On the one hand, when $p$ is non-degenerate with local model given by \eqref{non-degenerate}, we have:
\[ \var_{p}(\F) =  \bb_{p}(\F) = \frac{\lambda_{2}}{\lambda_{1}}+ \frac{\lambda_{1}}{\lambda_{2}} +2.\]
This implies our result, since  $\gsv_{p}(\F) = 0$.

On the other hand, for a saddle-node singularity, with  normal form as in \eqref{saddle-node-formal}, we have
\[ \var_{p}(\F) = k + 2 + \lambda \ \text{ and } \ \gsv_{p}(\F) = k , \]
while
\[ \bb_{p}(\F) = 2k +2 + \lambda.\]
\end{proof}

In the non-reduced  case, we have:
\begin{theorem}
\label{theorem-bb-zeta}
Let $\F$ be a germ of singular foliation at $(\C,p)$. Then
\[
\bb_{p}(\F) = \zeta_{p}(\F) + \sum_{q \in \cl{I}_{p}(\F)} \tau_{q}(\F)^{2},
\]
where the summation runs over all infinitely near points of $\F$ at $p$.
\end{theorem}
\begin{proof}
We recall  the behavior of the Baum-Bott index under blow-ups
\cite[Prop. 1]{brunella1997II}: if  $\sigma:(\tilde{\mathbb{C}}^2,D)\to (\C,p)$ is a blow-up at $p\in \C$ and  $q_1,\ldots,q_{\ell}$ are the $D$-singular points of
$\tilde{\F}$, then
\[ \bb_{p}(\F) = \sum_{j=1}^{\ell} \bb_{q_{j}}(\tilde{\F}) + c_{1}^{2}(N_{\tilde{\F}}).\]
This translates into
\begin{equation}
\label{baum-bott-blowup}
\bb_{p}(\F)=\begin{cases}
\displaystyle\sum^{\ell}_{j=1}\bb_{q_{j}}(\tilde{\F})+\nu^{2}_p(\F),  &\text{if $\sigma$ is non-dicritical}\\
\displaystyle\sum^{\ell}_{j=1}\bb_{q_{j}}(\tilde{\F})+(\nu_p(\F)+1)^2, &\text{if $\sigma$ is dicritical}.
\end{cases}
\end{equation}
Define
\[\vartheta_{p}(\F)=\bb_{p}(\F)-\zeta_{p}(\F) - \sum_{q \in \cl{I}_{p}(\F)} \tau_{q}(\F)^{2}.\]
We first observe that, if $\F$ is reduced, then $\cl{I}_{p}(\F) = \{p\}$ and $\tau_{p}(\F) = 0$,
resulting in $\vartheta_{p}(\F) = 0$ by the application of Lemma \ref{lemma_reduced}.
In general, if $\F$ is non-reduced, for a blow-up $\sigma$ as above,
we take into account the decomposition
\[\sum_{q \in \cl{I}_{p}(\F)} \tau_{q}^{2}(\F)  = \tau_{p}^{2}(\F)  + \sum_{j=1}^{\ell} \left(  \sum_{q \in \cl{I}_{q_j}(\tilde{\F})} \tau_{q}^{2}(\tilde{\F})  \right) \]
along with Propositions \ref{zeta-blowup}  and formula \eqref{baum-bott-blowup} in order to conclude that
\[\vartheta_{p}(\F) = \sum^{\ell}_{j=1}\vartheta_{q_{j}}(\tilde{\F}).\]
Finally, an induction argument gives that $\vartheta_{p}(\F) = 0$, proving the theorem.
\end{proof}

We recall that $\zeta_{p}(\F) = \var_{p}(\F) + \Delta_{p}(\F)$ and
 $\var_{p}(\F) =  \cs_{p}(\F) +\gsv_{p}(\F)$. When $\F$ is non-dicritical,
 $\Delta_{p}(\F)= \gsv_{p}(\F)$ and
  the theorem reads:
\begin{corollary}
\label{corollary-zeta-nondic}
If $\F$ is non-dicritical, then
\[
\bb_{p}(\F) = \cs_{p}(\F)+ 2 \gsv_{p}(\F) + \sum_{q \in \cl{I}_{p}(\F)} \tau_{q}(\F)^{2}.
\]
\end{corollary}

Since both $\Delta_{p}(\F)$  and $\gsv_{p}(\F)$ are integers, and $\gsv_{p}(\F) \geq 0$ when $\F$ is non-dicritical, the following corollary turns evident from Theorem \ref{theorem-bb-zeta}:
 \begin{corollary}
Let $\F$ be a germ of   foliation at $p\in\C$. Then
$$ \bb_{p}(\F)- \cs_{p}(\F) \in \ze.$$
This integer is non-negative when  $\F$ is non-dicritical.
 \end{corollary}
Indeed, this corollary could  also be  proved by following Baum-Bott indices along the reduction of
of $\F$ --- equation \eqref{baum-bott-blowup} --- and comparing them with $\cs$-indices for the reduced singularities.  Baum-Bott's Theorem (equation \eqref{baumbott}) brings the following consequence for global foliations:
 \begin{corollary}
Let $\F$ be a foliation on a compact surface $M$. Then
$$\sum_{p\in\sing(\F)}\cs_p(\F) \in \ze.$$
\end{corollary}

We have now all elements to complete the proof of Theorem \ref{main-theorem}:
\begin{proof}(of Theorem \ref{main-theorem})
The first statement follows straight from Theorem \ref{theorem-bb-zeta}.
If $\F$ is of second type at $p$,     so is it at all   infinitely near points, implying
 $\tau_{q}(\F) = 0$ for all $q \in \cl{I}_{p}(\F)$ and $\bb_{p}(\F) = \zeta_{p}(\F)
 = \var_{p}(\F) + \Delta_{p}(\F)$.
 Conversely, the equality of indices implies that the summation in Theorem  \ref{theorem-bb-zeta} vanishes,
 giving, in particular, that $\tau_{p}(\F) = 0$ and that $\F$ is of second type.
 The second statement is then a  consequence of Proposition \ref{total-delta-zero}.
\end{proof}

\section{Logarithmic foliations on the complex projective plane}
\label{logarithmic-section}
Let $\cl{F}$ be a holomorphic foliation on the complex projective plane $\pe^{2} = \pcn{2}$.  The degree of $\cl{F}$ is the number $\deg(\cl{F})$ of
tangencies  between $\cl{F}$ and a generic line. The question concerning the existence
of a bound for the degree $\deg(S)$ of an $\F$-invariant curve $S$ in terms of  $\deg(\cl{F})$ is known
 in foliation theory as   Poincar\'e problem \cite{poincare1891}. When all singularities of $\F$ over $S$ are non-dicritical, it is proven in
\cite{carnicer1994} that the inequality $\deg(S) \leq \deg(\cl{F}) + 2$ holds.
The limit case for this bound is reached by \textit{logarithmic foliations}, those defined
by logarithmic $1-$forms, as explained next.
Suppose that an $\F$-invariant   algebraic curve $S\subset {\mathbb P}^2 $ is defined by a homogeneous polynomial equation
$P=P_1P_2\cdots P_n=0$,
 where each polynomial $P_i$ is irreducible of degree $d_i$.  Suppose further that that  $\F$   is non-dicritical at each point  $q\in S$. Then the following statements are equivalent   \cite{brunella1997II,cano2015,cerveau2013}:
\begin{enumerate}
\item $\deg (S)=\deg({\mathcal F})+2$.

\item There are residues $\lambda_i\in {\mathbb C}^*$ with $\sum_{i=1}^n\lambda_id_i=0$ such that $\mathcal F$ is given by $\omega=0$, where $\omega$ is the global closed logarithmic $1$-form in ${\mathbb P}^2_{\mathbb C}$ defined by
    $$
    \omega=\sum_{i=1}^n\lambda_i \frac{dP_i}{P_i}.
    $$
\item The foliation ${\mathcal F}$ is a generalized curve foliation at each  $p \in \sing(\F) \cap S$ and $S$ contains all branches  of $\sep_{p}(\F)$ at each $p$.
\end{enumerate}

As an application of Theorem \ref{main-theorem}, we propose the following characterization of
non-dicritical logarithmic foliations:
\begin{proposition}
\label{prop-logarithmic}
Let $\F$ be a holomorphic foliation  on $\pe^{2}$. Suppose that $\F$ leaves invariant
an algebraic curve $S$  such that:
\begin{itemize}
\item
 $\sing(\F) \subset  S$;
\item all points $p \in \sing(\F)$ are non-dicritical and of second type;
\item $S$ contais all the local branches of $\sep_{p}(\F)$ at each $p \in \sing(\F)$.
\end{itemize}
Then $\deg (S)=\deg({\mathcal F})+2$ and $\F$ is a logarithmic foliation.
\end{proposition}
\begin{proof}
Denote by $d_{0} = \deg (S)$ and $d = \deg({\mathcal F})$.
On the one hand, by Baum-Bott's Theorem (equation \eqref{baumbott}), we have
\[\sum_{p \in \sing(\F)} \bb_{p}(\F) = (d + 2)^{2}.\]
On the other hand, by Theorem \ref{main-theorem} and formulas
\eqref{cs-theorem} and \eqref{gsv-global-sum},
\[
\begin{array}{rcl}
\displaystyle \sum_{p \in \sing(\F)} \bb_{p}(\F) & = &\displaystyle \sum_{p \in \sing(\F)} \cs_{p}(\F) + 2  \gsv_{p}(\F) \medskip \\
& = & d_{0}^{2} + 2 \left((d + 2)d_{0} - d_{0}^{2}\right) \medskip \\
& =  & 2  (d + 2)d_{0} - d_{0}^{2}.
\end{array}
\]
These two equations give $d_{0} = d + 2$, which implies that $\F$ is logarithmic.
\end{proof}

Actually, Proposition \ref{prop-logarithmic} can be stated in a more general setting, in the spirit of
\cite{brunella1997II} and \cite{licanic2000},  switching
$\pe^{2}$ to a compact projective surface $M$ with Picard group ${\rm Pic}(M) = \ze$.
We need a definition:  a meromorphic $1-$form $\omega$   on a complex manifold $M$ is \textit{logarithmic}
if both $\omega$ and $d \omega$ have simple poles over $(\omega)_{\infty}$.  We can then state:
\begin{proposition}
\label{prop-logarightmic-picZ}
 Let $M$ be a compact projective surface  with Picard group ${\rm Pic}(M) = \ze$.
Let $\F$ be a holomorphic foliation on $M$ that leaves invariant
a compact curve $S$  satisfying the conditions listed in Proposition \ref{prop-logarithmic}.
Then $\F$ is induced by a closed logarithmic $1-$form having simple poles over $S$.
\end{proposition}
\begin{proof}
Summing up  $\bb_{p}(\F) = \cs_{p}(\F) + 2  \gsv_{p}(\F)$ over all $p \in \sing(\F)$,
we find
\[ c_{1}(N_{\F})^{2} = S \cdot S + 2(  c_{1}(N_{\F}) \cdot S  - S \cdot S) \Rightarrow ( c_{1}(N_{\F}) - c_{1}(\cl{O}_{S}))^{2} = 0 .\]
Since ${\rm Pic}(M) = \ze$, the line bundle $L = N_{\F}^{*} \otimes \cl{O}_{S}$ is trivial, that is,
$N_{\F}  = \cl{O}_{S}$. Now, the proof   follows the steps of Proposition 10 in \cite{brunella1997II}.
We have that $\F$ is induced by a  meromorphic $1-$form $\omega$ on $M$ with empty zero divisor and whose pole divisor is $S$ with order one. The comment preceding that result
 also works here:  if $\sigma$ is a blow-up at $p \in \sing(\F)$,
 then $\sigma^{*} \omega$ has a pole of first order over $\sigma^{-1}(S)$. This is because
  $\F$ is non-dicritical and
second type, implying  that $C$, the complete curve
of   separatrices at $p$, satisfies   $\nu_{0}(\F) = \nu_{0}(C) -1$   by Proposition \ref{prop:Equa-Ba}.  Finally,   taking
  $\pi: \tilde{M} \to M$  a desingularization for $S$, the curve
$\tilde{S} = \pi^{-1}(S)$ has normal crossings and $\tilde{\omega} = \pi^{*} \omega$ has a simple pole over $\tilde{S}$. Since $\tilde{S}$ is invariant by $\tilde{\omega}$,
the exterior derivative $d \tilde{\omega} $ also has a simple pole over  $\tilde{S}$.
 That is, $\tilde{\omega}$ is a logarithmic form and   Deligne's Theorem \cite{deligne1971} asserts that it is closed, giving  that $\omega$ is also closed.
 \end{proof}

\section{Examples}
\label{examples-section}
We present two examples that give a numerical illustration of our results.
\begin{example}[Suzuki's example]
Let $\F$ be the germ of foliation at $(\C,0)$ defined by $$\omega=(y^3+y^2-xy)dx-(2xy^2+xy-x^2)dy.$$
 $\F$ is a dicritical generalized curve foliation   having the transcendental first integral
 $$ \frac{x}{y}\exp\left( \frac{y(y+1)}{x}\right)$$
and  admitting no meromorphic first integral    \cite{suzuki1978}.
After one blow-up, the foliation is regular and has a unique leaf that is tangent to the exceptional divisor with   tangency order one. It corresponds to the unique isolated separatrix $B_{1}$. The transverse leaves
give rise to dicritical separatrices. Chose one of them and   denote by $B_{2}$ the corresponding dicritical separatrix. Then $\B = B_{1} + B_{2}$ is
a balanced divisor of separatrices.
 It   follows from \eqref{baum-bott-blowup} that
$$\bb_{0}(\F)=(\nu_0(\F)+1)^2=(2+1)^2=9.$$
The following   simple calculation follow from \eqref{beta_10}:
\[\begin{array}{ccc}
\left. \begin{array}{l}
\var_{0}(\F,B_1)= 6 \medskip \\
\var_{0}(\F,B_2)=3
\end{array} \right\} & \Rightarrow &
\var_{0}(\F) = \var_{0}(\F,B_1 +  B_2)=  9.\\
\end{array} \]
Since we are in the generalized curve case,  $\Delta_{0}(\F)=0$ and Theorem \ref{main-theorem}
is verified.
\begin{center}
\begin{figure}[t]
\includegraphics[scale=0.6]{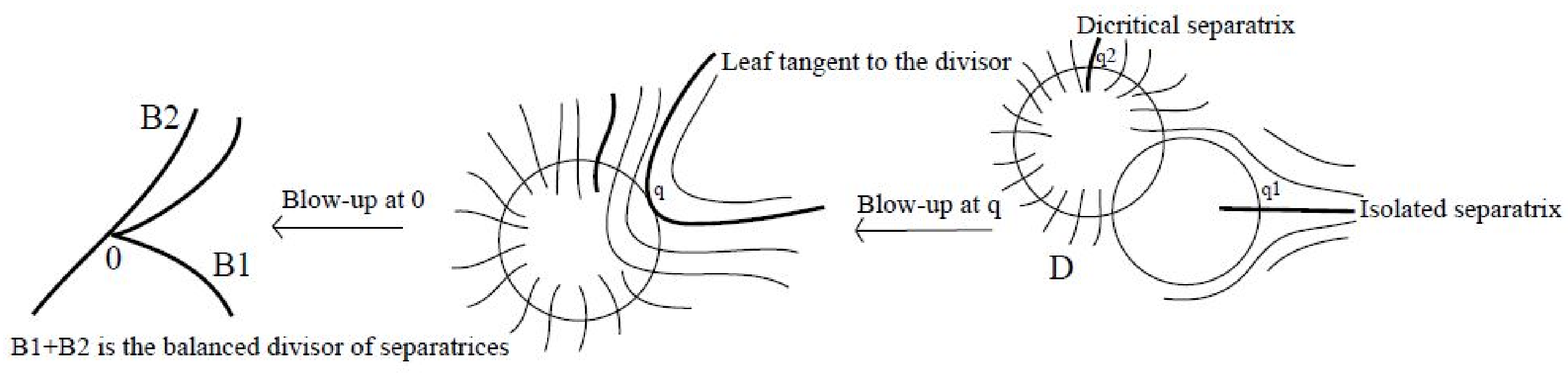}
\end{figure}
\end{center}

\end{example}

\begin{example}
Let $\F$ be the Ricatti foliation at $(\C,0)$ given by
$$\omega=(y^2+xy+x^2)dx+x^2dy.$$
  $\F$ is non-dicritical and has two separatrices $B_1:\{y=-x\}$ and $B_2:\{x=0\}$.
After one blow-up, the foliation has two reduced singularities.  The one corresponding to $B_1$, say $p_{1}$,   is a tangent saddle-node with weak index 2.   The other singularity, $p_{2}$,   is hyperbolic with eigenvalue ratio $-1$. Therefore $\F$ is not second type and $\tau_{0}(\F)=1$. The divisor $\B = B_{1} + B_{2}$ is a balanced one. Simple calculations using
\eqref{beta_10} lead to:
\begin{center}
\begin{figure}[b]
\includegraphics[scale=0.6]{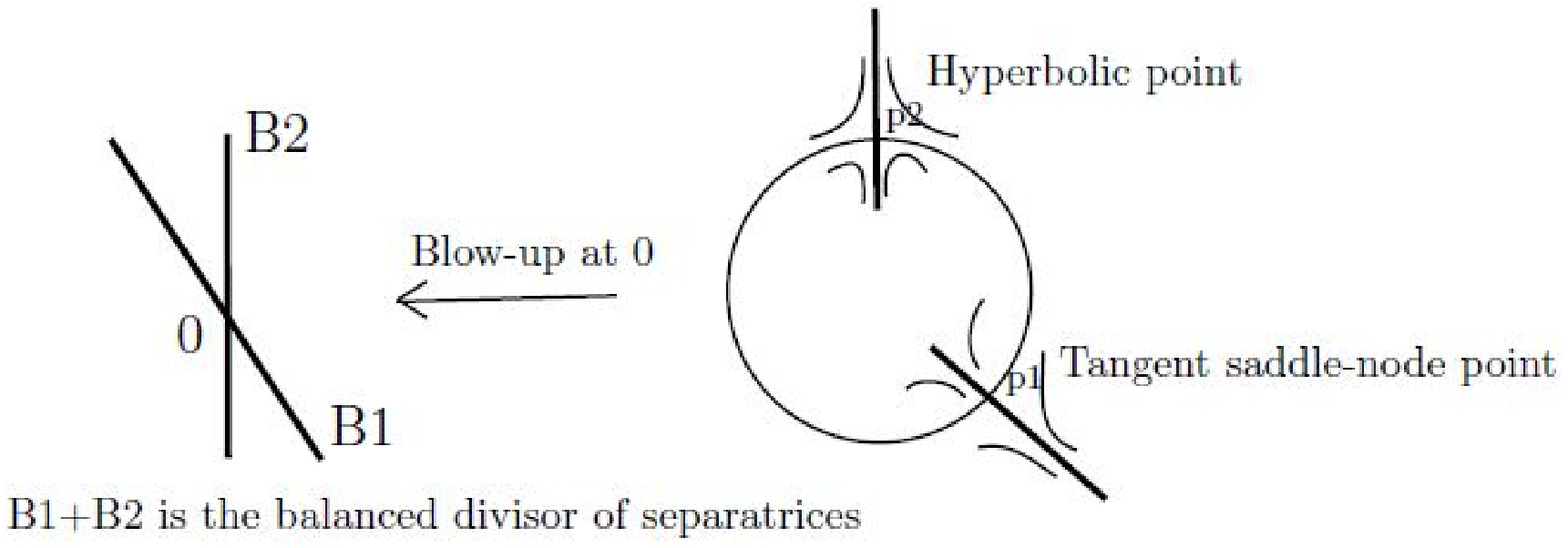}
\end{figure}
\end{center}
\[\begin{array}{ccc}
\left. \begin{array}{l}
\var_{0}(\F,B_1)= 3 \medskip \\
\var_{0}(\F,B_2)=2
\end{array} \right\} & \Rightarrow &
\var_{0}(\F) = \var_{0}(\F,B_1 +  B_2)=  5\\
\end{array} \]

\[\begin{array}{ccc}
\left. \begin{array}{l}
\Delta_{0}(\F,B_1)= 1 \medskip \\
\Delta_{0}(\F,B_2)=1
\end{array} \right\} & \Rightarrow &
\Delta_{0}(\F) = \Delta_{0}(\F,B_1 +  B_2)=  2.\\
\end{array} \]

The expression of Theorem \ref{theorem-bb-zeta} is   verified, since,
   from \eqref{baum-bott-blowup},
$$\bb_{0}(\F)=\bb_{p_1}(\F)+\bb_{p_2}(\F)+\nu_0(\F)^2=4+0+2^2=8.$$

\end{example}

\noindent{\it\bf Acknowledgements.}  The authors thank J.-P. Brasselet for early conversations that led to this article.

\bibliographystyle{plain}
\bibliography{referencias}

\begin{thebibliography}{10}

\bibitem{baum1972}
P.~Baum and R.~Bott.
\newblock Singularities of holomorphic foliations.
\newblock {\em J. Differential Geometry}, 7:279--342, 1972.

\bibitem{brunella1997I}
M.~Brunella.
\newblock Feuilletages holomorphes sur les surfaces complexes compactes.
\newblock {\em Ann. Sci. \'Ecole Norm. Sup. (4)}, 30(5):569--594, 1997.

\bibitem{brunella1997II}
M.~Brunella.
\newblock Some remarks on indices of holomorphic vector fields.
\newblock {\em Publ. Mat.}, 41(2):527--544, 1997.

\bibitem{camacho1984}
C.~Camacho, A.~Lins~Neto, and P.~Sad.
\newblock Topological invariants and equidesingularization for holomorphic
  vector fields.
\newblock {\em J. Differential Geom.}, 20(1):143--174, 1984.

\bibitem{camacho1982}
C.~Camacho and P.~Sad.
\newblock Invariant varieties through singularities of holomorphic vector
  fields.
\newblock {\em Ann. of Math. (2)}, 115(3):579--595, 1982.

\bibitem{cano2015}
F.~Cano, N.~Corral, and R.~Mol.
\newblock Local polar invariants for plane singular foliations.
\newblock {\em arXiv:1508.06882}, 2015.

\bibitem{carnicer1994}
M.~Carnicer.
\newblock The {P}oincar\'e problem in the nondicritical case.
\newblock {\em Ann. of Math. (2)}, 140(2):289--294, 1994.

\bibitem{cerveau2013}
D.~Cerveau.
\newblock Formes logarithmiques et feuilletages non dicritiques.
\newblock {\em J. Singul.}, 9:50--55, 2014.

\bibitem{deligne1971}
P.~Deligne.
\newblock Th\'eorie de {H}odge. {II}.
\newblock {\em Inst. Hautes \'Etudes Sci. Publ. Math.}, (40):5--57, 1971.

\bibitem{genzmer2007}
Y.~Genzmer.
\newblock Rigidity for dicritical germ of foliation in {$\Bbb C^2$}.
\newblock {\em Int. Math. Res. Not. IMRN}, (19):Art. ID rnm072, 14, 2007.

\bibitem{genzmer2016}
Y.~Genzmer and R.~Mol.
\newblock The poincar\'e problem in the dicritical case.
\newblock {\em arXiv:1608.07217}, 2015.

\bibitem{gomezmont1991}
X.~G{\'o}mez-Mont, J.~Seade, and A.~Verjovsky.
\newblock The index of a holomorphic flow with an isolated singularity.
\newblock {\em Math. Ann.}, 291(4):737--751, 1991.

\bibitem{khanedani1997}
B.~Khanedani and T.~Suwa.
\newblock First variation of holomorphic forms and some applications.
\newblock {\em Hokkaido Math. J.}, 26(2):323--335, 1997.

\bibitem{licanic2000}
S.~Licanic.
\newblock Logarithmic foliations on compact algebraic surfaces.
\newblock {\em Bol. Soc. Brasil. Mat. (N.S.)}, 31(1):113--125, 2000.

\bibitem{linsneto1988}
A.~Lins~Neto.
\newblock Algebraic solutions of polynomial differential equations and
  foliations in dimension two.
\newblock In {\em Holomorphic dynamics ({M}exico, 1986)}, volume 1345 of {\em
  Lecture Notes in Math.}, pages 192--232. Springer, Berlin, 1988.

\bibitem{martinet1982}
J.~Martinet and J.-P. Ramis.
\newblock Probl\`emes de modules pour des \'equations diff\'erentielles non
  lin\'eaires du premier ordre.
\newblock {\em Inst. Hautes \'Etudes Sci. Publ. Math.}, (55):63--164, 1982.

\bibitem{mattei2004}
J.-F. Mattei and E.~Salem.
\newblock Modules formels locaux de feuilletages holomorphes.
\newblock {\em arXiv:math/0402256}, 2004.

\bibitem{mol2002}
R.~Mol.
\newblock Meromorphic first integrals: some extension results.
\newblock {\em Tohoku Math. J. (2)}, 54(1):85--104, 2002.

\bibitem{mol2016}
R.~Mol and R.~Rosas.
\newblock Differentiable equisingularity of holomorphic foliations.
\newblock {\em arXiv:1611.03004}, 2016.

\bibitem{poincare1891}
H.~Poncar\'e.
\newblock Sur l'int\'egration alg\'ebrique des \'equations diff\'erentielles du
  premier ordre et du premier degr\'e.
\newblock {\em Rend. Circ. Mat. Palermo}, 5:161--191, 1891.

\bibitem{seidenbeg1968}
A.~Seidenberg.
\newblock Reduction of singularities of the differential equation
  {$A\,dy=B\,dx$}.
\newblock {\em Amer. J. Math.}, 90:248--269, 1968.

\bibitem{suwa1998}
T.~Suwa.
\newblock {\em Indices of vector fields and residues of singular holomorphic
  foliations}.
\newblock Actualit\'es Math\'ematiques. [Current Mathematical Topics]. Hermann,
  Paris, 1998.

\bibitem{suzuki1978}
M.~Suzuki.
\newblock Sur les int\'egrales premi\`eres de certains feuilletages analytiques
  complexes.
\newblock In {\em Fonctions de plusieurs variables complexes, {III} ({S}\'em.
  {F}ran\c{c}ois {N}orguet, 1975--1977)}, volume 670 of {\em Lecture Notes in
  Math.}, pages 53--79, 394. Springer, Berlin, 1978.

\end{thebibliography}

\end{document}